\documentclass[a4paper,reqno,10pt]{amsart}
\usepackage{t1enc} 
\usepackage{enumerate}
\usepackage[latin1]{inputenc}
\usepackage[english]{babel}
\usepackage{amsfonts} 
\usepackage{epsfig}
\usepackage{graphicx}
\usepackage{psfrag}
\usepackage{amsmath}
\usepackage{amsthm}
\usepackage{amssymb}
\usepackage{amscd} %diagrams
\usepackage{stmaryrd}
\usepackage{amsrefs}
\usepackage{bbm} %bold numbers
\usepackage{mathrsfs} %nice calligraphy

\usepackage{verbatim} %DO NOT comment this line

\newtheoremstyle{exercise} %for books or class notes
  {3pt} %space above
  {3pt} %space below
  {\scriptsize\rmfamily} %body font
  {
\parindent} %indent amount(empty=no indent,\parindent=para indent)
  {\rmfamily\scshape} %thm head font
  {.} %punctuation after thm head
  {.5em} %space after thm head: " " = normal interword space;
  {} %thm head spec (can be left empty, meaning `normal')

\newtheoremstyle{newplain}
  {5pt}
  {5pt}
  {\itshape}
  {}
  {\rmfamily\scshape}
  {. ---}
  {.5em}
  {}

\newtheoremstyle{newremark}
  {5pt}
  {5pt}
  {\rmfamily}
  {}
  {\rmfamily\scshape}
  {. ---}
  {.5em}
  {}

\theoremstyle{newplain}
\newtheorem*{Theorem*}{Theorem} %no numbering for Theorem*
\newtheorem*{thmglobal}{Theorem \ref{globalregularity}}
\theoremstyle{newplain}
\newtheorem{Theorem}{Theorem}
\newtheorem{Lemma}[Theorem]{Lemma}
\newtheorem{Corollary}[Theorem]{Corollary}
\newtheorem{Proposition}[Theorem]{Proposition}

\newtheorem{Definition}[Theorem]{Definition}

\theoremstyle{newremark}
\newtheorem{Remark}[Theorem]{Remark}

\theoremstyle{exercise}

\numberwithin{Theorem}{section}
\numberwithin{Exercise}{section}

\def\e{\varepsilon}

\def\F1{{\mathscr F}}
\def\S{\Sigma}

\def\pf{\begin{proof}[{\bf Proof:}]\thst}
\def\thst{\mbox{}\newline}

\newcommand{\N}{\mathbb{N}} %natural numbers
\newcommand{\R}{\mathbb{R}} %real numbers

\newcommand{\Rn}{\R^n}
\newcommand{\Hm}[1]{\mathscr{H}^{#1}} %Hausdorff measure

\newcommand{\cl}[1]{\mathcal{#1}}

\newcommand{\calA}{\mathscr{A}}

\newcommand{\calG}{\mathscr{G}}

\newcommand{\calN}{\mathscr{N}}

\newcommand{\calS}{\mathscr{S}}

\newcommand{\calU}{\mathscr{U}}

\def\XXint#1#2#3{{%
\setbox0=\hbox{$#1{#2#3}{\int}$}
\vcenter{\hbox{$#2#3$}}\kern-.5\wd0}}

\newcommand{\la}{\langle}
\newcommand{\ra}{\rangle}

\renewcommand{\leq}{\leqslant}
\renewcommand{\geq}{\geqslant}

\begin{document}

%=================
% TITLE AND AUTHOR
%=================
\title[Neumann free boundary singularities]{On the singular set of mean curvature flows with Neumann free boundary conditions}
\author{Amos N. Koeller}
\email{akoeller@everest.mathematik.uni-tuebingen.de}
\date{\today}

\address{Mathematisches Institut der Universit\"at T\"ubingen \\ 
Auf der Morgenstelle 10 \\
72076 T\"ubingen \\
Germany \\}

%=========
% ABSTRACT
%=========

\begin{abstract}
We consider $n$-dimensional hypersurfaces flowing by mean curvature flow with Neumann free boundary conditions supported on a smooth support surface. We show that the Hausdorff $n$-measure of the singular set is zero. In fact, we consider two types of interaction between the support and flowing surfaces. In the case of weaker interaction, we need make no further assumptions than in the case without boundary to achieve our result. In the case of stronger interaction, we need only make the additional assumption that $H_{\Sigma}>0$, that is, that the support surface be mean convex. We go on, in this case, to show that the result is not, in general, true without the mean convexity assumption.
\end{abstract}

\maketitle

%===========
% DEDICATION
%===========

\begin{comment}
\begin{flushright}
\textsf{\textit{Dedicated to: }}
\end{flushright}
\end{comment}

%==================
% TABLE OF CONTENTS
%==================

%==============================
% THE MEAT --- OR SO ONE THINKS
%==============================

\section{Introduction}

A time dependent family of surfaces, $\cl{M}=(M_t)$, is said to be moving by its mean curvature if for each time, $t$, and each point, $x\in M_t$, $x$ is moving at a velocity equal to its mean curvature along the unit normal at that point. The mean curvature flow has been extensively studied, both in the classical form, see, for e.g., Ecker \cite{ecker1} and Huisken \cite{huisken}, and in the weak form, the so called Brakke flow, see Brakke \cite{brakke}.

Of particular interest in both cases is the singular set, $sing_T\cl{M}$. That is, the set of points, reached by the flow, where the flow is no longer appropriately defined or collapses upon itself. It has been shown in the case of an $n$-dimensional surface flowing without boundary in $\R^{n+1}$, see Brakke \cite{brakke} or Ecker \cite{ecker1}, that the Hausdorff $n$-measure of the singular set is zero:
\begin{equation}\label{base}
\Hm{n}(sing_T\cl{M})=0. 
\end{equation}
Also of interest has been the study of mean curvature flow with boundary conditions. In particular, mean curvature flow with Neumann free boundary conditions, see, for e.g., Buckland \cite{buckland} and Stahl \cite{stahl1} and \cite{stahl}. Neumann free boundary conditions prescribe a fixed support surface, $\Sigma$, along which the surface is allowed to flow provided that the flowing and support surfaces always meet orthogonally. 

In this paper we consider the singular set of mean curvature flows with Neumann free boundary conditions with two differing interpretations of the nature of the support surface. In each case we give equivalent results to ($\ref{base}$). 

Firstly, we interpret the support surface as solid. That is, any intersection of the flowing surface and the support surface other than on the boundary of the flowing surface is treated as a singular point, having `collided' with the support surface.

Secondly, we interpret the support surface as traversable. That is, that the support surface should be only thought of as guidelines for the movement of the boundary of the flowing surface, but not actually present itself.

Initial results for the solid boundary case were given in Koeller \cite{koeller4}. The results were, however, dependent on  several assumptions. 

In this paper we remove the unwanted assumptions. In the case of traversable boundary, we show that we need only make the same assumptions as those needed for the case of mean curvature flow without boundary. In the solid boundary case, we make the additional assumption that the support surface is mean convex (actually something slightly more general). We also show however, that without this assumption, the equivalent to ($\ref{base}$) will not, in general, be true.

In summary, our main result states in simple terms:

{\it For any mean curvature flow with Neumann free boundary conditions, $\cl{M}$, supported on a mean convex support surface in the case that the surface is solid, $(\ref{base})$ holds.}

This result is stated formally in Section 3, as Theorem $\ref{globalregularity}$, after all the necessary terms have been properly defined.

The structure of this paper is as follows. In Section 2 we define the objects to be considered; the flow and the singular set. In Section 3, we present the assumptions made; the area continuity and unit density hypothesis, used also in the case without boundary, and the mean convexity of the support surface. It is also in Section 3 that the main theorem is stated.

In Section 4, we outline the strategy of the proof. In presenting the strategy we also introduce several supporting results that will be applied in this work. Section 4 also shows that almost all points are well behaved in a sense that is there defined.

In section 5 we give local curvature estimates in neighbourhoods around the well behaved points in the boundary of the limiting surface of the flow. These estimates are the technical key to our results.

In section 6 we use the local curvature estimates firstly to provide local regularity results around the well behaved points. That is, that ($\ref{base}$) holds in an appropriate form in small neighbourhoods of the well behaved points. With covering arguments, we then deduce the proof of the main theorem.

Finally, in Section 7, we discuss the necessity of the mean convexity of the support surface and show how to find the counter examples.

\section{Definitions and aims}
We begin by providing a formal definition of the problem being observed, namely mean curvature flows with Neumann free boundary conditions. We first define the support surface for the boundary and then how a surface is understood to flow on this support surface.

\begin{Definition}\label{rollingball}
Let $S\subset \R^{n+1}$ be a $C^1$-hypersurface and $\nu(x)$ be a choice of unit normal for each $x\in S$. $S$ is said to satisfy the rolling ball condition of radius $r>0$ if $B_r(x\pm r\nu(x))\cap S = \{x\}$ for each $x\in S.$
\end{Definition} 

\begin{Definition}  (Free boundary support surface) \label{freebdrysupport}
Let $G$ be a simply connected $C^3$-$(n+1)$-dimensional subset of $\R^{n+1}$. Let $\Sigma := \partial G$ satisfy the 
rolling ball condition for balls of radius $1/\kappa_{\Sigma}$ and satisfy the condition on the second fundamental 
form, $A_{\Sigma}$, of $\Sigma$
$$\parallel A_{\Sigma}\parallel^2 + \parallel\nabla A_{\Sigma}\parallel \leq \kappa_{\Sigma}^2 < \infty.$$
$\Sigma$ is then said to be a  Neumann free boundary support surface. 
\end{Definition} 
\begin{Remark}
$\Sigma$ will always denote the support surface of the flows being observed.
\end{Remark}
We now define the flows being considered in this work. The difference between the two being the role that the support $\Sigma$ takes. We start with the initial surface.
\begin{Definition}(Initial surface)\label{initialsurface}
Let $M^n$ denote a smooth orientable $n$-dimensional manifold with smooth, compact boundary, $\partial M^n$, and set 
$M_0:=F_0(M^n),$
where $F_0$ is a smooth embedding satisfying 
\begin{eqnarray}
\partial M_0 := F_0(\partial M^n) & \subset & M_0 \cap \Sigma \hbox{ and} \nonumber \\
\langle \nu_0,\nu_{\Sigma}\rangle (F_0(p))& = &0  \ \ \hbox{  } \ \ \hbox{ for all } p\in \partial M^n, \label{m0eqns}
\end{eqnarray}
for smooth unit normal fields $\nu_0$ to $M_0$ and $\nu_{\Sigma}$ to $\Sigma$. For $\nu_{\Sigma}$ we take the inner unit normal vector field to $G$.
\end{Definition} 

\begin{Definition}{ (Mean curvature flow with Neumann free boundary conditions)}\label{MCFwNfBC} 
Let $\Sigma$ be a Neumann free boundary support surface. 
Let $T\in (0,\infty)$, $I :=[0,T)$ be an interval, and $F(\cdot,t)~:~M^n\rightarrow\R^{n+1}$ be a one-parameter family of smooth 
embeddings for all $t \in I$. The family of hypersurfaces $\cl{M}:=(M_t)_{t\in I}$, where $M_t = F_t(M^n)$, are said to be 
evolving by { mean curvature with Neumann free boundary conditions on the solid support surface $\Sigma$} if 
\begin{equation}
\begin{array}{rcll}
\frac{\partial F}{\partial t}(p,t) & = & \vec{H}(p,t) & \hbox{ for all }(p,t) \in M^n \times I, \\
F(\cdot,0) & = & F_0, &   \\
\partial M_t := F(\partial M^n,t)& =& M_t\cap \Sigma & \hbox{ for all } t\in I,   \\
\langle \nu,\nu_{\Sigma} \rangle (F(p),t)&  = & 0 & \hbox{ for all }(p,t) \in \partial M^n \times I, \hbox{ and} \\
M_t & \subset & \overline{G} & \hbox{ for all } t\in I. 
\end{array}
\label{MCFwNfBCeqns}
\end{equation}
$\cl{M}:=(M_t)_{t\in I}$ are said to be evolving by { mean curvature with Neumann free boundary conditions on the traversable support surface $\Sigma$} if 
\begin{equation}
\begin{array}{rcll}
\frac{\partial F}{\partial t}(p,t) & = & \vec{H}(p,t) & \hbox{ for all }(p,t) \in M^n \times I, \\
F(\cdot,0) & = & F_0, &   \\
\partial M_t := F(\partial M^n,t)& \subset & M_t\cap \Sigma & \hbox{ for all } t\in I,   \hbox{ and}\\
\langle \nu,\nu_{\Sigma} \rangle (F(p),t)&  = & 0 & \hbox{ for all }(p,t) \in \partial M^n \times I,  
\end{array}
\label{ghosteqns}
\end{equation}
Here $\vec{H}(p,t)=-H(p,t)\nu(p,t)$ denotes the mean curvature vector of the immersions $M_t$ at $F(p,t)$, for a choice 
of unit normal $\nu$ for $M_t$.
\end{Definition} 
\begin{Remark} 
\begin{enumerate}[(1)]
\item For convenience of reference we will in general simply say that $\cl{M}$ is a $MCF(N,S)$ if it is a solution of ($\ref{MCFwNfBCeqns}$) and a $MCF(N,T)$ if it is a solution of ($\ref{ghosteqns}$). We will say that $\cl{M}$ is a $MCF(N)$ when whether $\cl{M}$ is a $MCF(N,T)$ or a $MCF(N,S)$ is not important to the discussion.
\item The solid support case is the case where we give a physical interpretation to $\Sigma$, making $\Sigma$ solid. In this case the flowing surface may not pass through $\Sigma$, but rather will collide with $\Sigma$ and cause the flow to cease.

The traversable support case is the case where the support surface is not to be thought of as a physical object, but rather simply a prescription of where the boundary, $(\partial M_t)_{t\in I}$, should flow. In this case the flowing surface may traverse $\Sigma$ without any consequences or special treatment.

\item We will, in general, suppress the notation referring to the embedding map, using rather only the 
position vector, $x \in \R^{n+1}$, instead of $F(p,t)$. With this understanding, we may re-express the above equations 
governing a $MCF(N,T)$ by
\begin{equation} \nonumber
\begin{array}{rcll}
\frac{\partial x}{\partial t} & = & \vec{H}(x) & \hbox{ for all } x\in M_t, \\
\partial M_t & \subset & \Sigma \cap M_t & \hbox{ for all } t\in I, \hbox{ and}\\
\langle \nu,\nu_{\Sigma}\rangle (x)& = & 0 & \hbox{ for all } x\in \partial M_t.
\end{array}
\end{equation}
We may also re-express the equations governing a $MCF(N,S)$ analogously.

\item That such flows exist, that is, that there are solutions to the system of equations ($\ref{ghosteqns}$), was proven by Stahl in \cite{stahl1} and \cite{stahl}. As solutions to the system of equations ($\ref{MCFwNfBCeqns}$) are a special case of ($\ref{ghosteqns}$), their existence also follows from Stahl's work. It follows that there is a maximal time, $T\in (0,\infty]$, for which there is a solution of ($\ref{MCFwNfBCeqns}$) or ($\ref{ghosteqns}$) over $I=[0,T)$. For the remainder of this work, when referring to a $MCF(N)$, that is, a solution $\cl{M}=(M_t)_{t\in [0,T)}$ of ($\ref{MCFwNfBCeqns}$) or ($\ref{ghosteqns}$), we will always use $T$ to refer to this maximal time. $T$ is also called the first singular time, as for all $t<T$, the surface $M_t$ can continue to flow and is therefore not singular. If $T=\infty$ the flow can always continue and there are therefore no singularities. We are therefore interested in the limit surface $M_T$ for $T<\infty$. We now work towards a formal definition of the singular set.
\end{enumerate}
\end{Remark}

\begin{Definition} \label{klausreach} 
Let $(M_t)$ be a one-parameter family of sets in $\R^{n+1}$. We say that the family (or flow, in the case that $(M_t)$ is a 
flow) reaches $x_0 \in \R^{n+1}$ at time $t_0$ if there exists a sequence $(x_j,t_j)$ with $t_j \nearrow t_0$ so that 
$x_j \in M_{t_j}$ and $x_j \rightarrow x_0$. We write $\cl{M} \rightarrow_{t_0} x_0$ to denote that $\cl{M}=(M_t)$ reaches $x_0$ at time $t_0$. 

For $\cl{M}=(M_t)_{t\in[0,T)}$, a mean curvature flow, we define the limit surface of $\cl{M}$, $M_T \subset \R^{n+1}$, by
$$M_T:=\{x\in \R^{n+1}:\cl{M} \rightarrow_{T} x\}.$$ For the boundary we first define 
$\partial \cl{M}:= (\partial M_t)_{t\in [0,T)}$
and analogously define the limit boundary $\partial M_T \subset \S$ of $\partial \cl{M}$ by
$$\partial M_T := \{x\in \S: \partial \cl{M} \rightarrow_{T} x_0\}.$$
\end{Definition} 
\begin{Remark} Due to the possible misunderstanding that the notation $\cl{M} \rightarrow_{t_0} x_0$ implies that $\cl{M}$ degenerates to 
the point $x_0$ we point out that this is not at all implied. $\cl{M} \rightarrow_{t_0} x_0$ simply denotes that $x_0$ is 
one of, in general, many points that are reached by the flow at time $t_0$. 

\end{Remark}

Using limit surfaces we are now also able to give a precise definition of the singular set.
\begin{Definition}\label{singregpoints}
Let $\cl{M}=(M_t)_{t\in [0,T)}$ be a $MCF(N,S)$. We say that $x_0\in \R^{n+1}$ is a regular point for $\cl{M}$ at time $t_0\in (0,T]$ if one of the following three conditions is satisfied:
\begin{enumerate}[(i)]
\item $x_0\not\in M_{t_0}$,
\item There exists a $\rho>0$ such that $B_{\rho}(x_0)\cap M_{t_0}$ is a smooth orientable properly embedded $n$-dimensional manifold and $B_{\rho}(x_0)\cap \Sigma =\emptyset$, or
\item There exists a $\rho>0$ such that $B_{\rho}(x_0) \cap M_{t_0}$ is a smooth orientable properly embedded $n$-dimensional manifold and that $\partial M_{t_0}\cap B_{\rho}(x_0)$ is a smooth $(n-1)$-dimensional manifold satisfying 
$$x_0 \in \partial M_{t_0} \cap B_{\rho}(x_0) = M_{t_0} \cap B_{\rho}(x_0) \cap \Sigma$$
and
$$\la\nu_{M_{t_0}}(x),\nu_{\Sigma}(x)\ra=0$$
for all $x\in \partial M_{t_0} \cap B_{\rho}(x_0)$, where $\nu_{M_{t_0}}$ is a choice of unit normal field for $M_{t_0}$.
\end{enumerate}
Let $\cl{M}=(M_t)_{t\in [0,T)}$ be a $MCF(N,T)$. We say that $x_0\in \R^{n+1}$ is a regular point for $\cl{M}$ at time $t_0\in (0,T]$ if one of the following three conditions is satisfied:
\begin{enumerate}[(i)]
\item $x_0\not\in M_{t_0}$,
\item There exists a $\rho>0$ such that $B_{\rho}(x_0)\cap M_{t_0}$ is a smooth orientable properly embedded $n$-dimensional manifold, or
\item There exists a $\rho>0$ such that $B_{\rho}(x_0) \cap M_{t_0}$ is a smooth orientable properly embedded $n$-dimensional manifold and that $\partial M_{t_0}\cap B_{\rho}(x_0)$ is a smooth $(n-1)$-dimensional manifold satisfying 
$$x_0 \in \partial M_{t_0} \cap B_{\rho}(x_0) \subset M_{t_0} \cap B_{\rho}(x_0) \cap \Sigma$$
and
$$\la\nu_{M_{t_0}}(x),\nu_{\Sigma}(x)\ra=0$$
for all $x\in \partial M_{t_0} \cap B_{\rho}(x_0)$, where $\nu_{M_{t_0}}$ is a choice of unit normal field for $M_{t_0}$.
\end{enumerate}

The set of all regular points of a $MCF(N,S)$, $\cl{M}$, at time $T$ is called the regular set of $\cl{M}$ which we denote by $reg_{T}^S\cl{M}$. If $x_0$ is not a regular point of $\cl{M}$ at time $T$ we say that it is a singular point of $\cl{M}$ at time $T$. The set of all singular points of $\cl{M}$ at time $T$ is called the singular set of $\cl{M}$ at time $T$ which we denote by $sing_{T}^S\cl{M}$. In the case that $\cl{M}$ is a $MCF(N,T)$ we replace the superscript $S$ with the superscript $T$. In the case that the nature of the flow is clear, or that $reg_T^S\cl{M}=reg_T^T\cl{M}$ respectively $sing_T^S\cl{M}=sing_T^T\cl{M}$, we omit the superscript.
\end{Definition}
It is within the above setting that we wish to prove a result in the form of $(\ref{base})$. That is,
\begin{equation}\label{mainresulteqn}
\Hm{n}(sing_T^J\cl{M})=0 \hbox{ for each } J\in \{S,T\}.
\end{equation} 
As noted in Koeller \cite{koeller4}, the above equation is, in the full generality just introduced, not true, at least for solutions of ($\ref{MCFwNfBCeqns}$), and we need therefore introduce assumptions.

\section{The assumptions and the main theorem}

That $(\ref{mainresulteqn})$ is not true in the general case for solutions of ($\ref{MCFwNfBCeqns}$) was shown in Koeller \cite{koeller4} by way of a counter example. (See Construction 3.1 in \cite{koeller4}.) Assumptions or conditions are therefore, unfortunately a necessity in this case. We keep our assumptions to a minimum, however, and make only two; one very natural and necessary assumption and the other of fundamental importance to the proof that is also made in the analysis of mean curvature flow without boundary.

In the case of traversable boundaries, that is of solutions to ($\ref{ghosteqns}$), it is, as with flows without boundary, not clear that ($\ref{mainresulteqn}$) is not in general true. However, just as in the case of solutions to ($\ref{MCFwNfBCeqns}$), our proof of this case is dependent on the results on mean curvature flows without boundary, and therefore on the technical assumption made there.

The first assumption in the solid boundary case is a response to the counter example mentioned above. In the counter example, the problematic part of the singular set arises from the interior of the flow reaching a part of the support surface which acts as an obstacle to the flow. We therefore first make an assumption to remove the possibility of obstacles arising in the support surface. In particular, we make the following assumption.
\begin{Definition}\label{b2b}
A $MCF(N,S)$, $\cl{M}=(M_t)_{t\in I}$, is said to satisfy the boundary approaches boundary assumption if 
\begin{equation}\label{b2beqn}
M_T\cap \Sigma = \partial M_T.
\end{equation}
That is, if $M_T\cap \Sigma=\{x\in \R^{n+1}:\partial \cl{M} \rightarrow_T x\}.$
\end{Definition}
As the boundary approaches boundary assumption is an unusual one, we immediately note the following important result.
\begin{Proposition}\label{convexb2b}
Suppose $\cl{M} = (M_t)_{t\in I}$ is a $MCF(N,S)$ supported on the support surface $\Sigma$ for which the condition 
\begin{equation}\label{assump1}
H_{\Sigma}(x)>0
\end{equation}
 is satisfied for each $x\in \Sigma$. Then, for any $t_0\leq T$, $\cl{M}_t\rightarrow_{t_0}x_0$ implies that $\partial \cl{M} \rightarrow_{t_0}x_0$.
\end{Proposition}
\begin{Remark}
\begin{enumerate}[(1)]
\item A proof of Proposition $\ref{convexb2b}$ can be found in \cite{koeller4}, Proposition 3.2.
\item Since, as shown in Section 7, the set of support surfaces not satisfying $(\ref{assump1})$ for which $(\ref{mainresulteqn})$ is not true is dense (with a type of $C^2$-metric) in the set of support surfaces not satisfying $(\ref{assump1})$, this is an essentially necessary assumption. 

\item Convexity and mean convexity assumptions are also natural ones, leaving a still large and interesting class of flows, that have been used in the literature by, for example, Stone \cite{stone} and Stahl \cite{stahl1} and \cite{stahl}.
\item We finally note that the above Proposition is important as it is a condition on the initial data and can therefore be reasonably checked. Checking the boundary approaches boundary assumption directly requires knowledge of the behaviour of the limit of the flow, which is not always easy, or possible, to obtain. However, since the boundary approaches boundary assumption is more general, it is this condition that we will continue to refer to in the remainder of this work.
\end{enumerate}
\end{Remark}

The second assumption is one used in the case without boundary. The assumption, or rather, hypothesis, is fundamental to the works of Brakke \cite{brakke} and Ecker \cite{ecker1} in their analysis of mean curvature flows without boundary. As we show the regularity of the interior of our flow by application of the analysis for flows without boundary, we also need to assume the same hypothesis. Named the area continuity and unit density hypothesis, the hypothesis is also instrumental in our analysis of the regularity of the boundary.

\begin{Definition} \label{areacontassump}
Let $\cl{M}=(M_t)_{t\in [0,T)}$ be a $MCF(N)$. $\cl{M}$ is said to satisfy the {area continuity and unit density hypothesis at time $T$} if the hypersurfaces $M_t$ converge in the sense of measures to a $\Hm{n}$-measurable, 
countably $n$-rectifiable subset $M_T $ of $\R^{n+1}$ of locally finite $\Hm{n}$-measure. That is,
\[ \lim_{t\nearrow T}\int_{M_t}\psi d\Hm{n} = \int_{M_T}\psi d\Hm{n}\]
for all $\psi \in C^0_0(\R^{n+1})$.
\end{Definition}
\begin{Remark}
Note that in this work we understand a countably $n$-rectifiable set to be any subset of $\R^{n+1}$, $M$, which can be expressed as 
\[ M \subset M_0 \cup \bigcup_{i=1}^{\infty}F_i(\Rn),\]
where $\Hm{n}(M_0)=0$ and the $F_i:\Rn \rightarrow \R^{n+1}$ are Lipschitz functions.
\end{Remark}

Having stated our assumptions, we are now in a position to give a precise statement of our main theorem.
\begin{Theorem}\label{globalregularity}(Main regularity theorem)
Let $\cl{M}=(M_t)_{t\in I}$ be either
\begin{enumerate}[(i)]
\item a $MCF(N,S)$ satisfying the boundary approaches boundary assumption, or
\item a $MCF(N,T)$,
\end{enumerate}
that satisfies the area continuity and unit density hypothesis. then
\begin{equation}\label{mainthmeqn}
\Hm{n}(\partial M_T)=0 \hbox{ and }\Hm{n}(sing_T\cl{M})=0.
\end{equation}
\end{Theorem}

\section{Strategy and Regularity}
To prove Theorem $\ref{globalregularity}$, we break $sing_T\cl{M}$ into several parts and then consider each one separately. 

We first notice that 
$$\Hm{n}(sing_T\cl{M}) = \Hm{n}(sing_T\cl{M} \cap \partial M_T)+\Hm{n}(sing_T\cl{M} \sim \partial M_T)$$
(where $\sim$ denotes set subtraction). We may then immediately deal with $\Hm{n}(sing_T\cl{M} \sim \partial M_T)$ by noting that away from $\partial M_t$ we may use localising arguments to apply the results on mean curvature flows without boundary.

\begin{Lemma}\label{interior}
Let $\cl{M}=(M_t)_{t\in [0,T)}$ be either 
\begin{enumerate}[(i)]
\item a $MCF(N,S)$ satisfying the boundary approaches boundary assumption, or
\item a $MCF(N,T)$,
\end{enumerate}
satisfying the area continuity and unit density hypothesis. Then
$$\Hm{n}(sing_T\cl{M} \sim \partial M_T)=0.$$
\end{Lemma}
\begin{proof}
Consider firstly case (i). Corollaries 4.6 and 8.1 in Koeller \cite{koeller4} imply that 
$$\Hm{n}(sing_T\cl{M} \sim \Sigma)=0,$$
as for any $x_0\not\in M_T$, $\cl{M}\not\rightarrow_Tx_0$ and thus $x_0$ is regular. Since, by the boundary approaches boundary assumption, $\partial M_T = M_T\cap \Sigma$, the result follows.

In case (ii), we note that for any $x_0 \not\in \partial M_T$, there exists a $\rho >0$ such that $(B_{\rho}(x_0) \cap M_t)_{t\in I}$ is a mean curvature flow without boundary. It now follows, again from Corollaries 4.6 and 8.2 in \cite{koeller4} that
$$\Hm{n}(sing_T\cl{M}\sim \partial M_T)=0.$$
\end{proof}
It follows from Lemma $\ref{interior}$ that, both in the solid and traversable boundary cases, it is sufficient to consider the Hausdorff measure of $\partial M_T$.

An immediate application of the above Lemma allows us to reduce our attention to just one of the two cases, namely the traversable boundary case. The solid boundary case will then follow.
\begin{Corollary}\label{justTcase}
If, for each $MCF(N,T)$ satisfying the area continuity and unit density hypothesis, $\cl{M}=(M_t)_{t\in [0,T)}$, $\Hm{n}(\partial M_T)=0$, then for any $MCF(N,S)$ satisfying the area continuity and unit density hypothesis and the boundary approaches boundary assumption, $\cl{M}^*=(M_t)_{t\in [0,T^*)}$,
$$\Hm{n}(\partial M_{T^*})= 0 \hbox{ and } \Hm{n}(sing_T\cl{M}^*)=0.$$
\end{Corollary}
\begin{proof}
Let $\cl{M}=(M_t)_{t\in [0,T_S)}$ be a $MCF(N,S)$ satisfying the area continuity and unit density hypothesis and the boundary approaches boundary assumption. Then, directly from the definitions of the flows with solid and traversable boundary, $\cl{M}$ is also a $MCF(N,T)$ with $T_S$ smaller than or equal to the maximal time, $T$, for which the flow may be continued when considered as a $MCF(N,T)$. That is, we may extend the flow $\cl{M}$ to $\cl{M}':=(M_t)_{t\in [0,T)}$, $T\geq T_S$, such that $\cl{M}'$ is a $MCF(N,T)$ with first singular time $T$.

Now, for each $t<T$, $\partial M_t=F_t(\partial M^n)$ for some smooth proper embedding $F_t$ on the smooth basis manifold $M^n$ as given in Definition $\ref{initialsurface}$. It follows that $\Hm{n}(\partial M_t)=0$. Furthermore, if $T=T_S$, then  $\cl{M}'=\cl{M}$ is a $MCF(N,T)$ satisfying the area continuity and unit density hypothesis, so that by the hypothesis of the Theorem, $\Hm{n}(\partial M_T)=0$. It now follows that $\Hm{n}(\partial M_{T_S})=0$ so that, by Lemma $\ref{interior}$, it now follows that 
$$\Hm{n}(sing_{T_S}\cl{M})\leq \Hm{n}(\partial M_{T_S})=0.$$
\end{proof}
\begin{Remark}
We deduce from Lemma $\ref{interior}$ and Corollary $\ref{justTcase}$ that it is sufficient to prove that $\Hm{n}(\partial M_T)=0$ for any $MCF(N,T)$. This is our aim for the remainder of the paper. From this point on, therefore, unless otherwise specified, any references to a flow, $\cl{M}=(M_t)_{t\in I}$, will refer to a $MCF(N,T)$.
\end{Remark}
To consider $\Hm{n}(\partial M_T)$ we break the set $\partial M_T$ up into further smaller parts.

Firstly, we recall from standard geometric measure theory (see, for e.g., Federer \cite{federer} or Simon \cite{simon1}) the following result.

\begin{Theorem}\label{rectifiablecharacteristics1}
Let $M\subset \R^{n+1}$ be a countably $n$-rectifiable set of locally finite measure. Then, for $\Hm{n}$-almost all $x\in \R^{n+1}$ either
\begin{enumerate}[(I)]
\item $\Theta^n(\Hm{n},M,x):=\lim_{\rho \searrow 0}\frac{\Hm{n}(B_{\rho}(x)\cap M)}{\omega_n\rho^n}=0,$ or 
\item the approximate tangent space, $T_xM$, of $M$ at $x$ exists. That is, 
\[\lim_{\lambda\searrow 0}\int_{M^{x,\lambda}}\phi d\Hm{n} = \int_{T_xM}\phi d\Hm{n} \]
for all $\phi\in C_0^0(\R^{n+1})$, 
\end{enumerate}
where $M^{x,\lambda}=\lambda^{-1}(M-x)$ for $\lambda > 0$. 

In particular, for any $MCF(N,T)$ satisfying the area continuity and unit density hypothesis, $\cl{M}= (M_t)_{t\in [0,T)}$, either (I) or (II) holds for $\Hm{n}$-almost all $x\in \R^{n+1}$ with $M$ replaced by $M_T$.
\end{Theorem}
For a preselected $MCF(N,T)$, $\cl{M}=(M_t)_{t\in [0,T)}$, and for each selection $J\in \{I, II\}$, we define 
\begin{equation}\label{rectpoints1}
R_J:=\{x\in \R^{n+1}:\hbox{ Theorem } \ref{rectifiablecharacteristics1} (J) \hbox{ holds at }x \hbox{ with } M=M_T\}.
\end{equation}
To prove Theorem $\ref{globalregularity}$ it now follows, from Theorem $\ref{rectifiablecharacteristics1}$, that we need only consider points in $\partial M_T \cap R_J$ for $J\in \{I,II\}$. Such points are, however, still not necessarily easy to work with. We therefore introduce Ecker's (\cite{ecker1}) good points, which are points around which the area of the flow behaves well toward the limit.
\begin{Definition}\label{goodpoints}
Let $\cl{M}=(M_t)_{t\in [0,T)}$ be a $MCF(N,T)$. For $\alpha \geq 0$ we define 
$$G_{T}^{\alpha}=
\left\{x\in \R^{n+1}:\limsup_{\rho\searrow 0}\frac{1}{\rho^n}\int_{T-\rho^2}^T
\int_{M_t \cap B_{\rho}(x)}|\vec{H}|^2d\Hm{n}\leq \alpha^2\right\} $$
and
$$\calG:=\bigcap_{\alpha \geq 0}G_{T}^{\alpha}.$$
We say that $x\in \R^{n+1}$ is a good point if $x\in \calG$.
\end{Definition}
We may restrict our attention to good points, since almost all points are good points. 
\begin{Lemma}\label{nobadpoints}
Let $\cl{M}=(M_t)_{t\in [0,T)}$ be a $MCF(N,T)$. Then 
$$\Hm{n}(\R^{n+1} \sim \calG)=0.$$
\end{Lemma}
\begin{Remark}
The proof of Lemma $\ref{nobadpoints}$ is as in Lemma 7.6 in Koeller \cite{koeller4}. 
\end{Remark} 
It follows from Lemma $\ref{nobadpoints}$ that, in order to prove Theorem $\ref{globalregularity}$, it remains only to show that
$$\Hm{n}(\partial M_T \cap R_J\cap \calG)=0$$
for $J\in \{I,II\}$ and flows $\cl{M}=(M_t)_{t\in [0,T)}$. For $J=I$ the following Theorem provides the desired result.
\begin{Theorem}\label{pointI}
Let $\cl{M}$ be a $MCF(N,T)$ and $x\in \Sigma \cap R_J\cap \calG$. Then $\cl{M}\not\rightarrow_T x$. In particular, $x\not\in \partial M_T$ and 
$$\Hm{n}(\partial M_T \cap R_I\cap \calG)=0.$$
\end{Theorem}
\begin{Remark}
The proof of Theorem $\ref{pointI}$ is as in Corollary 8.1 in Koeller \cite{koeller4} and Lemma 15.5 in Ecker \cite{ecker1}.
\end{Remark}
We conclude that in order to prove Theorem $\ref{globalregularity}$ it remains only to show, for any $MCF(N,T)$, $(M_t)_{t\in [0,T)}$, that
\begin{equation}\label{finalterm}
\Hm{n}(\partial M_T \cap R_{II}\cap \calG)=0.
\end{equation}
It is the proof of $(\ref{finalterm})$ that is the technical heart of this paper. The proof is presented in the following two sections. 
\begin{Remark}
The proofs of Lemma 7.6 and Corollary 8.1 in Koeller \cite{koeller4} are actually stated for mean curvature flows with Neumann free boundary conditions with a solid boundary satisfying the boundary approaches boundary condition. The statements and proofs, however, are identical in our present case and we therefore do not repeat them here.
\end{Remark}

\section{Local curvature bounds}
In this section we show that for any $MCF(N,T)$, $\cl{M}=(M_t)_{t\in [0,T)}$, and $x_0\in \partial M_T$ there is a radius $\rho>0$ such that the second fundamental form of $M_t$ is uniformly bounded in $B_{\rho}(x_0)\times [T-\rho^2,T)$. 

The proof is accomplished by firstly using the existence of the tangent plane to show an integral height excess decay result. This is then used to give an absolute height excess decay result. That is, that in sufficiently small balls and on relatedly small sized time intervals, the surfaces, as sets, are very near the tangent plane. 

This allows us to show that, in some smaller ball and time interval, the second fundamental form, and therefore curvature, remain bounded. That the local regularity result, $\Hm{n}(\partial M_T\cap R_{II}\cap \calG \cap B_{\rho}(x_0))=0$ holds, and therefore that Theorem $\ref{globalregularity}$ also holds, can be deduced from these local curvature estimates. The proof is given in the following section.

To start, we note firstly that our analysis in this section uses so called blow-up arguments regularly. That is, we analyse the surfaces $M_t$ under parabolic rescaling, which we define below.
\begin{Definition}\label{paradef}
Let $\cl{M}=(M_t)_{t\in [t_1,T)}$ be a $MCF(N,T)$, $\lambda>0$, $x_0\in \R^{n+1}$, and $t_0\in (0,T]$. The parabolic rescaling, or blow-up, of $\cl{M}$ by a factor of $\lambda$ around $(x_0,t_0)$ is the one-parameter family of smooth, properly embedded hypersurfaces
$$(M_s^{(x_0,t_0),\lambda})_{s\in [-\lambda^{-2}t_0,0)}$$
where
$$M_s^{(x_0,t_0),\lambda}:=\lambda^{-1}(M_{\lambda^2s+t_0}-x_0).$$
That is, $(M_s^{(x_0,t_0),\lambda})_{s\in [-\lambda^{-2}t_0,0)}$, is the result of the application of the change of variables
\begin{equation}\label{CoV}
y=\lambda^{-1}(x-x_0) \hbox{ and } s=\lambda^{-2}(t-t_0)
\end{equation}
to $\cl{M}$.
\end{Definition}
\begin{Remark}
\begin{enumerate}[(1)]
\item It is standard theory that the blow-up
$$(M_s^{(x_0,t_0),\lambda})_{s\in [-\lambda^{-2}t_0,0)}$$ 
of a mean curvature flow, $\cl{M}=(M_t)_{[0,T)}$, continues to be a mean curvature flow. See, for e.g., Buckland \cite{buckland} or Ecker \cite{ecker1}. It follows that the blow-up of a solution, $\cl{M}=(M_t)_{[0,T)}$, of ($\ref{ghosteqns}$) is a solution of ($\ref{ghosteqns}$) supported on $\lambda^{-1}(\Sigma-x_0)$ over the interval $I=[-\lambda^{-2}T,0)$.

\item Should the centre of a blow up $(x_0,t_0)$ be clear, we will write $M_s^{\lambda}$ to refer to the parabolic rescaling $M_s^{(x_0,T),\lambda}$.
\end{enumerate}
\end{Remark}

Now, to realise our intention of deducing properties of $M_t$ from the existence of a tangent plane to $M_T$ at $x_0$, we need to show that they are in some way related. This is the purpose of considering only good points, for at good points we have the following convergence property.

\begin{Lemma}\label{integralest}
Let $\cl{M}=(M_t)_{t\in [0,T)}$ be a $MCF(N,T)$ satisfying the area continuity and unit density hypothesis. Then there is an $A<\infty$ such that, for each $\alpha \in (0,1/2]$ and $x_0\in G_T^{\alpha}$, there is a $\rho_0>0$ such that  
$$\sup_{t\in [T-\rho^2,T]}\left|\int_{M_t}\phi d\Hm{n} - \int_{M_T}\phi d\Hm{n}\right| \leq 2\alpha(\sup |\phi| + 
\sqrt{A}\rho\sup|D\phi|)\rho^n$$
holds for all $\rho\in (0,\rho_0]$ and $\phi \in C_0^1(B_{\rho}(x_0))$.
\end{Lemma}
\begin{Remark} 
\begin{enumerate}[(1)]
\item For the proof of Lemma $\ref{integralest}$ see Lemma 7.7 in Koeller \cite{koeller4}. Again this proof is for flows with solid boundary, which, however, remains unchanged for the case with traversable boundary.

\item We will actually apply Lemma $\ref{integralest}$ in its parabolically rescaled form which, by applying the change of variables $(\ref{CoV})$ for any $\lambda >0$, states
\begin{eqnarray}
\left|\int_{M_s^{(x_0,T),\lambda}}\phi  - \int_{M_T^{x_0,\lambda}}\phi \right| & = & \left|\int_{M_s^{(x_0,T),\lambda}}\phi  - \int_{M_0^{(x_0,T),\lambda}}\phi \right| \nonumber \\
& \leq & 2\alpha(\sup |\phi| + \sqrt{A}R\sup|D\phi|)R^n \label{intest}
\end{eqnarray}
for each $R\in (0,R_0]$, $s\in [-R_0^2,0]$, and $\phi \in C_C^1(B_R(0))$, where the integrals are taken with respect to $\Hm{n}$ and $R_0:=\lambda^{-1}\rho_0$.
\end{enumerate}
\end{Remark}

We now show the local height estimates for flows $\cl{M}=(M_t)_{t\in [0,T)}$ around points in $\partial M_T \cap R_{II} \cap \calG$. We note that here, and in the remainder of the work, $\pi_T:\R^n\rightarrow T$ denotes the orthogonal projection onto $T$. Furthermore, letting $\{e_1,...,e_{n+1}\}$ denote the standard basis for $\R^{n+1}$, we identify $span\{e_1,...,e_j\}$ with $\R^j$ for $1\leq j <n+1$ and write $x_i$ to denote the $i$th component of $x$, $\la x,e_i\ra$. We also note that for $1\leq j <n+1$, we write $B_r^j(x)$ to denote $B_r(x)\cap \R^j\subset \R^{n+1}$.

\begin{Lemma}\label{N1}
Let $\cl{M}=(M_t)_{t\in [0,T)}$ be a $MCF(N,T)$ satisfying the area continuity and unit density hypothesis and let $x_0\in \partial M_T \cap R_{II} \cap \calG.$

Then, for each $\e>0$, there exists $\rho_0=\rho_0(\e)>0$ such that 
$$\sup_{t\in (T-\rho_0^2,T)}\int_{M_t\cap B_{\rho}(x_0)}\left|\pi_{T_{x_0}^{\perp}M_T}(x-x_0)\right|^2d\Hm{n} \leq \e \rho^{n+2} \hbox{ for all } \rho\in (0,\rho_0].$$
\end{Lemma}
\begin{proof}
Without loss of generality we may assume that $x_0=0$ and $T_{x_0}M_T=\R^n$ so that we need to show the existence of $\rho_0=\rho_0(\e)>0$ such that 
\begin{equation}\label{N1e1}
\sup_{t\in (T-\rho_0^2,T)}\int_{M_t\cap B_{\rho}(x_0)}x_{n+1}^2d\Hm{n} \leq \e \rho^{n+2} \hbox{ for all } \rho\in (0,\rho_0].
\end{equation}
Supposing that ($\ref{N1e1}$) is not true, then there exists a sequence  $\rho_j \searrow 0$ and a sequence $t_j\in (T-\rho_j^2,T)$ such that 
\begin{equation}\label{N1e2}
\int_{M_{t_j}\cap B_{\rho_j}(0)}x_{n+1}^2 d \Hm{n}>\e \rho_j^{n+1} \hbox{ for each } j\in \N.
\end{equation}
Define now $\tilde{\phi}:\R^{n+1} \rightarrow \R$ by $\tilde{\phi}(x)=x_{n+1}^2$, and choose $\psi \in C_C^{\infty}(B_2(0))$ such that 
$$\psi\geq 0, |D\psi|\leq 4, \hbox{ and } \phi \equiv 1 \hbox{ on } B_1(0).$$
Then $\phi:=\psi\tilde{\phi}\in C_C^1(B_2(0))$ and $|D\phi|\leq 16$.

Define further $\lambda_j=\rho_j$ and $s_j=\lambda_j^{-2}(T-t_j) \in (-1,0)$. Then clearly, as $T_0M_T=\R^n$ and $\phi=0$ on $\R^n$, 
$$\lim_{j\rightarrow \infty}\int_{\lambda_j^{-1}M_T}\phi d \Hm{n}=0.$$
Also, by reversing the parabolic change of variables centred around $(x_0,T)$ and by ($\ref{N1e2}$), we see that 
\begin{eqnarray}
\int_{M_{s_j}^{\lambda_j}}\phi   \geq  \int_{M_{s_j}^{\lambda_j}\cap B_1(0)}x_{n+1}^2  
 =  \int_{M_{t_j}\cap B_{\rho_j}(0)}\rho_j^{-n-2}x_{n+1}^2  
 >  \e, \nonumber
\end{eqnarray}
where the integrals are taken with respect to $\Hm{n}$, for each $j\in \N$. Thus
\begin{equation}\label{N1e3}
\left| \int_{M_{s_j}^{\lambda_j}}\phi d\Hm{n} - \int_{\lambda_j^{-1}M_T} \phi d\Hm{n}\right|>\frac{\e}{2}
\end{equation}
for all sufficiently large $j\in \N$.

However, as $x_0 \in \calG \subset G_T^{\alpha}$ with $\alpha = \e 2^{-n-2}(1+32\sqrt{A})$, we see, by Lemma $\ref{integralest}$, that 
$$\left|\int_{M_{s_j}^{\lambda_j}}\phi d\Hm{n} -\int_{\lambda_j^{-1}M_T}\phi d\Hm{n}\right|\leq 2\alpha (1+2\sqrt{A}16)2^n < \frac{\e}{2}.$$
This contradiction to ($\ref{N1e3}$) proves the result.
\end{proof}

Lemma $\ref{N1}$ gives height estimates in a weak, that is, integral sense. We need to deduce strong height estimates, that is estimates on the supremum of the height of all points. We make this deduction by combining Lemma $\ref{N1}$ and the clearing out Lemma which we recall below. The idea being that, by Lemma $\ref{N1}$, any points with large heights will be part of a narrow peak with little surface area; the clearing out Lemma then ensures that such points quickly recede from the summit, so that a short time later, no points have large height.

\begin{Lemma}\label{clearingout} 
There exists a constant $\kappa_n=\kappa_n(\kappa_{\Sigma},n)$ such that if $\cl{M}=(M_t)_{t\in [0,T)}$ is a $MCF(N,T)$ and $\cl{M}\rightarrow_{t_0}x_0$ for some $t_0\in (0,T]$ and $x_0\in \R^{n+1}$, then, for any $\beta\in (0,1/2n)$ there exists  a constant $\theta = \theta(n,\beta)\in (0,1/2)$ such that for all $\rho\in (0,\kappa_n]$
$$\rho^{-n}\Hm{n}(M_{t_0-\beta\rho^2}\cap B_{\rho}(x_0))\geq \theta.$$
Equivalently, if for some $\rho \in (0,\kappa_n)$ and $\beta \in (0,1/2n)$
$$\rho^{-n}\Hm{n}(M_{t_0-\beta\rho^2}\cap B_{\rho}(x_0))< \theta,$$
then there exists $\e>0$ such that 
$$M_t \cap B_{\e}(x_0)=\emptyset$$
for all $t\in (t_0-\e^2,t_0)$. That is, $\cl{M}\not\rightarrow_{t_0}x_0$.
\end{Lemma}
\begin{Remark}
The clearing out Lemma for flows without boundary is due to Brakke \cite{brakke}. A proof that is directly applicable to the interior points in our case can be found in Proposition 4.23 in Ecker \cite{ecker1}. The proof for the case that $x_0\in \partial M_T$ is identical to that given for flows with solid boundary as given in Corollary 6.10 in Koeller \cite{koeller4}. Clearly, by taking the minimum of the constants $\theta$ found in the proof of the interior case and the boundary case, we can find a constant $\theta=\theta(n,\beta)$ that holds in both cases.
\end{Remark}
\begin{Lemma}\label{N2}
Let $0<c_0<1/2$ (where $\kappa_n$ is as stated in Lemma $\ref{clearingout}$), then there exists an $\e_0>0$ such that for any $MCF(T,N)$, $\cl{M}=(M_t)_{t\in [0,T)}$, satisfying the area continuity and unit density hypothesis, and any $x_0 \in M_{T}$ for which $T_{x_0}M_T$ exists and
\begin{equation}\label{N2e01}
\sup_{t\in (T-\rho^2,T)}\int_{M_t\cap B_{\rho}(x_0)}\left|\pi_{T_{x_0}^{\perp}M_T}(x-x_0)\right|^2d\Hm{n} < \e_0 \rho^{n+2}
\end{equation}
for all $0<\rho \leq \rho_0< \min\{\kappa_n, T^{1/2}\}$, we have
\begin{equation}\label{N2e02}
\sup_{t\in [T-\rho^2/4,T]}\sup_{x\in M_t\cap \overline{B_{\rho /2}(x_0)}}\left|\pi_{T_{x_0}^{\perp}M_T}(x-x_0)\right|^2 \leq c_0^2\rho^2
\end{equation}
for all $\rho\in (0,\rho_0]$.
\end{Lemma}
\begin{proof}
We may assume that $x_0=0$ and $T_{x_0}M_T=\R^n$. We note that it is sufficient to show that should ($\ref{N2e01}$) hold for some given $\rho\in (0,\rho_0]$, then ($\ref{N2e02}$) holds for that same $\rho$. Moreover, by otherwise parabolically rescaling, we may assume that $\rho=1$. We also note that under such a parabolic rescaling $c_0<1/2<1\leq \rho_0 \leq \kappa_n.$  

Suppose now that the claim is not true. Then, for each $j\in \N$, we can find a flow $\cl{M}^j:=(M_t^j)_{t\in [0,T)}$ in $MCF(N,T)$ satisfying the area continuity and unit density hypothesis such that 
$$\sup_{t\in (T-1,T)}\int_{M_t^j\cap B_{1}(0)}x_{n+1}^2d\Hm{n} < j^{-1}$$
but that 
$$\sup_{t\in [T-1/4,T]}\sup_{x\in M_t^j\cap \overline{B_{1/2}(x_0)}}x_{n+1}^2 > c_0^2.$$
For each $j\in \N$, let $t_j \in [T-1/4,T]$ be a time at which 
\begin{equation}\label{N2e1}
\sup_{x\in M_{t_j}^j \cap \overline{B_{1/2}(0)}} x_{n+1}^2>c_0^2.
\end{equation}
We now consider $y\in \overline{B_{1/2}(0)} $ with $y_{n+1}^2 \geq c_0^2$ and calculate
\begin{eqnarray}
2^nc_0^{-n}&\Hm{n}&(M_{t_j-(1/4n)(c_0/2)^2}\cap B_{c_0/2}(y)) \nonumber \\
& \leq & 2^{n+2}c_0^{-n-2}\int_{M_{t_j-(1/4n)(c_0/2)^2}\cap B_{c_0/2}(0)}x_{n+1}^2 d\Hm{n} \nonumber \\
&\leq & 2^{n+2}c_0^{-n-2}j^{-1}. \nonumber
\end{eqnarray}
Taking $j_0$ so large that $2^{n+1}c_0^{-n-2}j_0^{-1} < \theta(n,1/4n)$, where $\theta$ is as given in Lemma $\ref{clearingout}$, we deduce from Lemma $\ref{clearingout}$ that there is an $\e_y>0$ such that $M_t\cap B_{\e_y}(y)=\emptyset$ for all $t\in (t_{j_0}-\e_y^2,t_{j_0})$. In particular, we deduce that $y\not\in M_{t_{j_0}}$. By the choice of $y$ it follows that 
$$\overline{B_{1/2}(0)}\cap \{y\in \R^{n+1}:y_{n+1}^2\geq c_0^2\}\cap M_{t_{j_0}}=\emptyset$$
contradicting ($\ref{N2e1}$). The result follows.
\end{proof}
A simple extension to Lemma $\ref{N2}$ states that the result can be formulated to hold with the centre of the surface being estimated permitted to be taken anywhere within some small neighbourhood of $(x_0,T)$.
\begin{Proposition}\label{N4}
Let $\cl{M}:=(M_t)_{t\in [0,T)}$ be a $MCF(N,T)$ and $x_0 \in M_T$.
If 
\begin{equation}\label{N4e1}
\sup_{t\in [T-\rho^2,T]}\sup_{x\in M_t \cap \overline{B_{\rho}(x_0)}}|\pi_{T_{x_0}^{\perp}M_T}(x-x_0)|^2 < \e_0\rho^2 \hbox{ for all } \rho\in (0,\rho_0]
\end{equation}
for some $\e_0>0$ and $0<\rho_0<T^{1/2}$, then
\begin{equation}\label{N4e2}
\sup_{t\in [\tau-\rho^2/4,\tau]}\sup_{x\in M_t \cap \overline{B_{\rho/2}(y)}}|\pi_{T_{x_0}^{\perp}M_T}(x-y)|^2 < 4\e_0\rho^2 
\end{equation}
for all $y\in M_{\tau}\cap \overline{B_{\rho/2}(x_0)}, \tau \in [T-\rho^2/4,T]$ and $\rho \in (0,\rho_0]$.
\end{Proposition}
\begin{proof}
We may assume, as in the preceding results, that $x_0=0$ and $T_{x_0}M_T=\R^n$. Furthermore it is sufficient to show that if ($\ref{N4e1}$) holds for some $\rho\in (0,\rho_0]$, then ($\ref{N4e2}$) holds for that same $\rho$. Suppose now that $\tau \in [T-\rho^2/4,T]$, $y\in M_T\cap \overline{B_{\rho/2}(0)}$ and that $x\in M_{t_x} \cap \overline{B_{\rho/2}(y)}$ for some $t_x\in [T-\rho^2/4,T]$.

Then
$$(x-y)_{n+1}^2\leq x_{n+1}^2+y_{n+1}^2 + 2x_{n+1}y_{n+1} \leq 4(\max\{x_{n+1}^2,y_{n+1}^2\}).$$
Noting that $x,y\in \overline{B_{\rho/2}(0)}$ and that 
$$t_x,\tau \in [T-2\rho^2/4 ,T] \subset [T-\rho^2/2, T],$$ 
we deduce from ($\ref{N4e1}$) that $\max\{x_{n+1}^2,y_{n+1}^2\}\leq \e_0\rho^2$ and the result follows.
\end{proof}
Using the above results, we can now show that local curvature bounds exist around points in $\partial M_T \cap R_{II} \cap \calG$. We present the estimate in the following lemma and corollary, which are based on the local regularity theorem presented as Theorem 5.7 in Ecker \cite{ecker1}.

\begin{Lemma}\label{N5}
There exist $\e_0, c_0>0$ such that the following holds.

Suppose that $\cl{M}=(M_t)_{t\in [0,T]}$ is a $MCF(N,T)$ for which
$$\sup_{t\in [\tau-\rho^2,\tau]}\sup_{x\in M_t\cap \overline{B_{\rho}(y)}}x_{n+1}^2 \leq \e_0\rho^2$$
for each $\tau\in [T-\rho^2,T]$, $y\in \overline{B_{\rho}(0)}$ and $\rho\in (0,\rho_0]$ for some $\rho_0^2<T/2$.

Then
$$\sup_{t\in [T-\rho_0^2/4,T]} \sup_{x\in M_t\cap \overline{B_{\rho_0/2}(0)}}|A_{M_t}(x)|^2\leq c_0\rho_0^{-2}.$$
\end{Lemma}
\begin{proof}
Suppose that the statement is not true, then there exists a sequence of flows, $\cl{M}^j:=(M_t^j)_{t\in [T-\rho_j^2,T]}$, each a $MCF(N,T)$, for which 
$$\sup_{t\in [\tau-\rho^2,\tau]}\sup_{x\in M_t^j\cap \overline{B_{\rho}(y)}}x_{n+1}^2\leq \rho^2j^{-2}$$
for all $\tau\in [T-\rho^2,T]$, $y\in \overline{B_{\rho}(0)}$ and $\rho\in [0,\rho_j]$, but that 
$$\sup_{t\in [T-\rho_j^2/4,T]}\sup_{x\in M_t^j\cap \overline{B_{\rho_j/2}(0)}}|A_{M_t^j}(x)|^2\rho_j^2\rightarrow \infty.$$
Parabolically rescaling so that $T=0$ and $\rho_j=1$ for each $j\in \N$, we have a sequence $\cl{M}^j=(M_t^j)_{t\in [-1,0]}$, each a $MCF(N,T)$, for which
\begin{equation}\label{N5e1}
\sup_{t\in [\tau-\rho^2,\tau]}\sup_{x\in M_t^j\cap \overline{B_{\rho}(y)}}x_{n+1}^2\leq \rho^2j^{-2}
\end{equation}
for all $\tau\in [-1,0]$, $y\in \overline{B_{1}(0)}$ and $\rho\in [0,1]$, but that 
$$\sup_{t\in [T-1/4,T]}\sup_{x\in M_t^j\cap \overline{B_{1/2}(0)}}|A_{M_t^j}(x)|^2\rightarrow \infty.$$
Since
\begin{eqnarray}
\gamma_j^2 & := & \sup_{\sigma\in [0,1]}\sigma^2 \sup_{t\in [-(1-\sigma)^2,0]}\sup_{x\in M_t^j \cap \overline{B_{1-\sigma}(0)}}|A_{M_t^j}(x)|^2 \nonumber \\
& \geq & \sup_{t\in [-1/4,0]}\sup_{M_t^j\cap \overline{B_{1/2}(0)}}|A_{M_t^j}(x)|^2,\nonumber
\end{eqnarray}
$$\lim_{j\rightarrow \infty}\gamma_j^2=\infty.$$
By the hypothesis of the smoothness of the $\cl{M}^j$ up to and including $t=0$, however, we see that $\gamma_j<\infty$ for each given $j\in \N$.

For each given $j\in \N$ we can now find  $\sigma_j\in (0,1]$, $\tau_j\in [-(1-\sigma_j)^2,0]$, and $y_j \in \overline{B_{1-\sigma_j}(0)}$ such that 
$$\gamma_j^2=\sigma_j^2|A_{M_{\tau_j}^j(y_j)}|^2.$$
We deduce that 
$$\sigma_j^2\sup_{[-(1-\sigma_j/2)^2,0]}\sup_{M_t^j \cap \overline{B_{1-\sigma_j/2}(0)}} |A_{M_t^j(x)}|^2 \leq 4\gamma_j^2$$
so that 
$$\sup_{[-(1-\sigma_j/2)^2,0]}\sup_{M_t^j \cap \overline{B_{1-\sigma_j/2}(0)}} |A_{M_t^j(x)}|^2 \leq 4|A_{M_{\tau_j}^j(y_j)}|^2,$$
and thus 
$$\sup_{[\tau_j-\sigma_j^2/4,\tau_j]}\sup_{M_t^j \cap \overline{B_{\sigma_j/2}(y_j)}} |A_{M_t^j(x)}|^2 \leq 4|A_{M_{\tau_j}^j(y_j)}|^2,$$
as
$$\overline{B_{\sigma_j/2}(y_j)}\times [\tau_j-\sigma_j^2/4,\tau_j]\subset \overline{B_{1-\sigma_j/2}(0)} \times[-(1-\sigma_j/2)^2,0].$$
Now let $\lambda_j=|A_{M_{\tau_j}}^j(y_j)|^{-1}$ and define  $(\tilde{\cl{M}}^j)=(\tilde{M}^j_s)_{s\in [-\lambda_j^{-2}\sigma_j^2/4,0]}$ to be the parabolic rescaling of the flow $(M_t^j)_{t\in [\tau_j-\sigma_j^2/4,\tau_j]}$ by a factor of $\lambda_j$ around $(y_j,\tau_j)$.

Then, for each $j\in \N$, $(\tilde{M}^j_s)$ is a $MCF(N,T)$ satisfying
\begin{equation}\label{N5e2}
0\in \tilde{M}^j_0, \ \ |A_{\tilde{M}_0^j(0)}|=1
\end{equation}
and 
\begin{equation}\label{N5e3}
\sup_{s\in[-\lambda_j^{-2}\sigma_j^2/4,0]}\sup_{x\in\tilde{m}_s^j \cap \overline{B_{\lambda_j^{-1}\sigma_j/2}(0)}}|A_{\tilde{M}_s^j}(x)|^2\leq 4.
\end{equation}
Since $\lambda_j^{-2}\sigma_j^2=\gamma_j^2 \rightarrow \infty$ we deduce that 
$$\sup_{s\in [-R^2,0]}\sup_{x\in \tilde{M}_s^j \cap \overline{B_R(0)}}|A_{\tilde{M}_s^j}(x)|^2 \leq 4$$
for each $R>0$ and sufficiently large $j$ depending on $R$. Parabolically rescaling inequality ($\ref{N5e1}$) around $(y_j,\tau_j)$ by a factor of $\lambda_j$, we get 
$$\sup_{s\in [-\lambda_j^{-2}\rho^2,0]}\sup_{x\in \tilde{M}_s^j \cap \overline{B_{\lambda_j}^{-1}\rho}}|\lambda_jx_{n+1}|\leq \rho j^{-1}$$
for each $j\in \N$ and $\rho\in (0,1]$.

For fixed $R>0$, we set $\rho=R\lambda_j$. As $\lambda_j \rightarrow 0$ as $j\rightarrow \infty$, we see that $\rho<1$ for sufficiently large $j$ and thus 
\begin{equation}\label{N5e4}
\sup_{s\in [-R^2,0]}\sup_{x\in \tilde{M}_s^j \cap \overline{B_R(0)}}|x_{n+1}| \leq Rj^{-1}
\end{equation}
for such $j$.

The curvature estimates in ($\ref{N5e2}$) and ($\ref{N5e3}$) imply, by the Arzela-Ascoli Theorem, that we may take a smooth limit of $\tilde{M}^j$ to find a $MCF(N,T)$, $(M_s^{\prime})_{s\leq 0}$ satisfying 
\begin{equation}\label{N5e5}
0\in M_0^{\prime}, \ \ |A_{M_0^{\prime}}(0)|=1,
\end{equation}
and $|A_{M_s^{\prime}}(y)|^2\leq 4$ for all $s\leq 0$ and $y\in M_s^{\prime}$. However, by ($\ref{N5e4}$) we also have $|x_{n+1}|=0$ for all $x\in M_s^{\prime}$ and $s\leq 0$. Thus $M_s^{\prime}=\R^n$ for each $s\leq 0$ and hence $|A_{M_0^{\prime}}(0)|=0$. This contradiction to ($\ref{N5e5}$) proves the result.
\end{proof}

\begin{Corollary}\label{N6}
Let $\cl{M}=(M_t)_{t\in [0,T)}$ be a $MCF(N,T)$ satisfying the area continuity and unit density hypothesis and let 
$$x_0\in \partial M_T \cap R_{II} \cap \calG.$$ 
Then there exists a radius $\rho_0>0$ and a constant, $c_1$, such that 
$$\sup_{t\in (T-\rho_0^2,T)}\sup_{x\in M_t \cap B_{\rho_0}(x_0)}|A_{M_t}(x)|^2 \leq c_1\rho_0^{-2}.$$
\end{Corollary}
\begin{proof}
Without loss of generality, we may assume that $x_0=0$ and that $T_{x_0}M_T=\R^n$. Let $0<\e_1 < (1/16)\min\{\e_0,1/2\}$, where $\e_0$ is as in Lemma $\ref{N5}$.

Then, as $x_0\in \partial M_T \cap R_{II} \cap \calG$, by Lemma $\ref{N2}$ and Proposition $\ref{N4}$ there is an $\e>0$ such that if 
\begin{equation}\label{N6e1}
\sup_{t\in (T-\rho^2,T)}\int_{M_t\cap B_{\rho}(0)}x_{n+1}^2 d\Hm{n} \leq \e\rho^{n+1}
\end{equation}
for all $\rho \in (0,2\rho_{\e}]$ for some $0<2\rho_{\e}\leq \min\{\kappa_n, T^{1/2}\}$,
$$\sup_{t\in [\tau -\rho^2/4 ,\tau]}\sup_{x\in M_t \cap \overline{B_{\rho/2}(y)}}x_{n+1}^2 < 4\e_1\rho^2<\frac{\e_0\rho^2}{4}$$
for all $y\in M_{\tau}\cap \overline{B_{\rho/2}(0)}$, $\tau \in [T-\rho^2/4,T]$ and $\rho \in (0,\rho_{\e}]$.

We deduce from Lemma $\ref{N1}$ that there is indeed a $\rho_{\e}>0$ such that ($\ref{N6e1}$) holds. It follows, for each $\delta \in (0,\rho_{\e}/8)$, that $(M_t)_{t\in [T-\delta^2 - \rho_{\e}^2/16,T-\delta^2]}$ is a $MCF(N,T)$ smooth up to and including $T-\delta^2$, which satisfies
$$\sup_{t\in [\tau-\rho^2,\tau]}\sup_{x\in M_t \cap \overline{B_{\rho}(y)}}x_{n+1}^2 < \e_0 \rho^2$$
for each  $y\in M_{\tau}\cap \overline{B_{\rho}(0)}$, $\tau\in [T-\rho^2,T]$ and $\rho\in (0,\rho_{\e}/4]$.

For each $\delta \in (0,\rho_{\e}/8)$, we infer from Lemma $\ref{N5}$ that 
$$\sup_{t\in [T-\delta^2 - \rho_{\e}^2/64 , T-\delta^2]}\sup_{x\in M_t \cap B_{\rho_{\e}/8}(0)}|A_{M_t}(x)|^2 \leq 16c_0\rho_{\e}^{-2}$$
and thus that 
$$\sup_{t\in [T-\rho_{\e}^2/64 , T)}\sup_{x\in M_t \cap B_{\rho_{\e/8}}(0)}|A_{M_t}(x)|^2 \leq 16c_0\rho_{\e}^{-2}.$$
Setting $\rho_0 = \rho_{\e}/8$ and $c_1 = c_0/4$ completes the proof.
\end{proof}
\section{Local and global regularity} 

With the local curvature bounds and the Arzela-Ascoli Theorem, we prove the regularity results, by showing that in small neighbourhoods of points in $\partial M_T \cap R_{II} \cap \calG$ we may take a limit of $M_t$ that is sufficiently smooth and, in particular, rectifiable. We deduce, with the use of Theorem $\ref{rectifiablecharacteristics1}$, that $\partial M_T$ has no $\Hm{n}$ measure in a small neighbourhood of almost all points. Using standard covering arguments, the main Theorem then follows.

\begin{Definition}\label{ptg}
Recall that $\nu_{\Sigma}$ is the inner unit normal field of $\Sigma$ with respect to $G$. We define $P_t^G$ to be the set of $x\in \partial M_t$ such that $\nu_{\Sigma}$ is the inner unit normal of $\partial M_t$ with respect to  $M_t$ and $P_t^c:=\partial M_t \sim P_t^G$.
We now define
$$M_t^G:=\overline{M_t\cap G}=(M_t\cap \overline{G})\sim P_t^c, M_t^c:=\overline{M_t\cap G^c}=(M_t \cap \overline{G^c})\sim P_t^G,$$
$$P_T^G:=\{x\in \R^{n+1}:P_t^G\rightarrow_T x\}, \hbox{ and }P_T^C:=\{x\in \R^{n+1}:P_t^c \rightarrow_Tx\}.$$
\end{Definition}
\begin{Remark}
We note that if $\cl{M}$ is a $MCF(N,S)$, then we have $M_t=M_t^G$, and hence $\partial M_t=P_t^G$ and $M_t^c=\emptyset$ for all $t\in [0,T)$. We also note that $\partial M_t = P_t^c \cap P_t^G$ for all $t\in [0,T)$ always holds.
\end{Remark}
\begin{Theorem}\label{N8}
Let $\cl{M}=(M_t)_{t\in [0,T)}$ be a $MCF(N,T)$ satisfying the area continuity and unit density hypothesis and $x_0\in \partial M_T \cap R_{II} \cap \calG$. Then there exists a $\rho_{x_0}>0$ such that 
$$\Hm{n}(B_{\rho_{x_0}}(x_0)\cap \partial M_T)=0.$$
\end{Theorem}
\begin{proof}By Corollary $\ref{N6}$, there exist $c_0,\rho_0>0$ such that 
\begin{equation}\label{N8e1}
\sup_{t\in [T-\rho_0^2,T)}\sup_{x\in B_{\rho_0}(x_0)}|A_{M_t}(x)|^2\leq \frac{c_0^2}{\rho_0^2}.
\end{equation}
We now consider $M_t^G \cap \overline{B_{\rho_0/2}(x_0)}$. By the curvature bounds, ($\ref{N8e1}$), we may, by considering local graph representations and using a diagonal argument to approach $\Sigma$, use the Arzela-Ascoli Theorem to conclude that 
$$B_{\rho_0/2}(x_0)\cap M_t^G \rightarrow M_0 \subset B_{\rho_0/2}(x_0) \cap G,$$
an immersed, $n$-dimensional $C^1$-manifold satisfying 
$$|A_{M_0}(x)|^2\leq c_0^2\rho_0^{-2}, B_{\rho_0/2}(x_0) \cap P_T^G \subset M_0, \hbox{ and}$$
\begin{equation}\label{N8e2}
\la \nu_{\Sigma},\nu_{M_0}\ra (x)=0 \hbox{ for all } x\in P_T^G \cap B_{\rho_0/2}(x_0).
\end{equation}
(As $M_0$ is `only' an immersed manifold, $\la \nu_{\Sigma},\nu_{M_0}\ra (x)=0$ is meant in the sense that there is a subset $M_0^s\subset M_0$ for which $\la \nu_{\Sigma},\nu_{M_0^s}\ra (x)=0$.)
Since $M_0\subset M_T$, it follows from the area continuity and unit density hypothesis that $M_0$ is a countably 
$n$-rectifiable set with locally finite $\Hm{n}$-measure.

We now consider $x\in P_T^G\cap B_{\rho_0/2}(x_0)$ and deduce from ($\ref{N8e2}$) that if $T_xM_0$ exists, then $\la T_{x_0}^{\perp}M_0,\nu_{\Sigma}(x)\ra =0$ and thus 
\begin{equation}\label{N8e3}
\nu_{\Sigma}(x)\in T_xM_0.
\end{equation}
Define $B_r^x:= B_{\kappa_{\Sigma}^{-1}}(x-r\nu_{\Sigma}(x))\subset int(G^c)$ and select 
$\varphi \in C_C^0(B_{\kappa_{\Sigma}^{-1}}^x-x)$ with $\varphi \geq 0$ and $\varphi =1$ on $B_{\kappa_{\Sigma}^{-1}/2}^x - x$.

By ($\ref{N8e3}$)
$$\int_{T_xM_0}\varphi d\Hm{n} >0.$$
However, since $B_r^x-x \subset int(G^c)-x$ and $M_0 \subset \overline{G}$
$$\int_{M_0^{x,\lambda}}\varphi d\Hm{n}=0 \hbox{ for all } \lambda >0.$$
It follows that $T_xM_0$ does not exist.

We now deduce from Theorem $\ref{rectifiablecharacteristics1}$ that $\Hm{n}(P_T^C \cap B_{\rho_0/2}(x_0))=0.$
An analogous argument shows that $\Hm{n}(P_T^c \cap B_{\rho_0/2}(x_0)) =0$
and therefore, since $\partial M_T = P_T^c\cup P_T^G$, that 
$$\Hm{n}(\partial M_T \cap B_{\rho_1}(x_0))=0$$
with $\rho_1=\rho_0/2$.
\end{proof}
\begin{Remark}
\begin{enumerate}
\item Examples can easily be constructed to show that, at least in the case of flows with traversable support surface, $M_0$ may indeed be `only' immersed instead of embedded. That is, there exist points $x\in \Sigma$ such that 
$$\partial M_t^G \sim \partial M_t \rightarrow_T x, \hbox{ and } \partial M_t \rightarrow_T x.$$
\item We may certainly also use the Arzela-Ascoli Theorem to take a limit surface of $M_t \cap B_{\rho_0/2}(x_0)$, as $t\rightarrow T$, as a whole. However, we would then need to exclude the possibility that there exist any points $x\in P_T^G \cap P_T^c$, as in this case $T_xM_0$ may, in fact, exist. Considering $M_T^G$ and $M_T^c$ separately avoids the potentially troublesome argumentation to handle this case directly.
\item That $M_0$ is countably $n$-rectifiable can also be shown without reference to the area continuity and unit density hypothesis using the $C^1$-properties of the surface. It was, however, as the area continuity and unit density hypothesis is also used elsewhere in the proof, convenient to use the hypothesis.
\item We could also, on a smaller ball, using the inner differentiability result of Stahl, \cite{stahl1}, show that the convergence is smooth to a smoothly immersed surface. The higher derivatives are, however, not necessary here.
\end{enumerate}
\end{Remark}  
As the measure of $\partial M_T$ has now been shown to be zero in a small ball around almost all points, we can now prove our main theorem, the global regularity result, Theorem $\ref{globalregularity}$, through covering arguments. Before presenting the proof, we restate the theorem for convenience.
\begin{thmglobal}
Let $\cl{M}=(M_t)_{t\in I}$ be either
\begin{enumerate}[(i)]
\item a $MCF(N,S)$ satisfying the boundary approaches boundary assumption, or
\item a $MCF(N,T)$,
\end{enumerate}
that satisfies the area continuity and unit density hypothesis. then
\begin{equation}
\Hm{n}(\partial M_T)=0 \hbox{ and }\Hm{n}(sing_T\cl{M})=0.
\end{equation}
\end{thmglobal}
\begin{proof}
By Lemma $\ref{interior}$, it suffices to show that $\Hm{n}(\partial M_T)=0,$
and thus, by Corollary $\ref{justTcase}$, to show that $\Hm{n}(\partial M_T)=0$
whenever $\cl{M}$ is a $MCF(N,T)$. We therefore assume that $\cl{M}$ is a $MCF(N,T)$. 

By Lemma $\ref{nobadpoints}$ and Theorems $\ref{rectifiablecharacteristics1}$ and $\ref{pointI}$
\begin{equation}\label{N9e1}
\Hm{n}(D:=\{x\in \partial M_T:x\not\in R_{II} \hbox{ or }x \not\in \calG\})=0.
\end{equation}
Let $R>0$, $\calU$ be any open covering of $D$, and $U:=\cup\{V:V\in \calU\}$. We consider 
$$\calA(\calU,R):=(\partial M_T \sim U)\cap \overline{B_R(0)},$$ a compact set. By Theorem $\ref{N8}$, for any $x\in \calA(\calU,R)$ there exists a $\rho_x>0$ such that $\Hm{n}(\partial M_T \cap B_{\rho_x}(x))=0$.

As $\calA(\calU,R)$ is compact, we can cover $\calA(\calU,R)$ by finitely many such balls $B_{\rho_x}(x)$ to deduce that 
$$\Hm{n}(\partial M_T \cap \calA(\calU,R))=0.$$
Letting $R\rightarrow \infty$ and defining $\calA(\calU):=\cup\{\calA(\calU , R):R>0\}$ it follows that $\Hm{n}(\partial M_T \cap \calA(\calU))=0.$

Now let $\e, \delta>0$. As $\Hm{n}(D)=0$ we can find an open $\delta$-covering, $\calU_{\delta , \e}:=\{B_i\}_{i=1}^{\infty}$, of $D$ satisfying 
$$b:=\sum_{i=1}^{\infty}\omega_n2^{-n}d(B_i)^n<\e$$
(where here $\omega_n$ is the Lebesgue measure of the unit $n$-ball). Define $U_{\delta , \e}:= \cup\{U:U\in \calU_{\delta , \e}\}$. As $\Hm{n}(\partial M_T \cap \calA(U_{\delta , \e}))=0$ we may similarly find an open $\delta$-covering of $\partial M_T \cap \calA(U_{\delta , \e})$, $\cl{A}:=\{U_i\}_{i=1}^{\infty}$ with 
$$a:=\sum_{i=1}^{\infty}\omega_n 2^{-n}d(U_i)^n < \e.$$
Thus $\Hm{N}_{\delta}(\partial M_T)\leq b+a < 2\e$. Letting $\delta , \e \rightarrow 0$ we deduce 
$$\Hm{n}(\partial M_T)=0.$$
\end{proof}

\section{The $H_{\Sigma}\leq 0$ case}
In this final section we consider the necessity of the assumption that $H_{\Sigma} > 0$. We have mentioned that examples exist showing that $(\ref{mainresulteqn})$ does not hold in general without a similar assumption to $H_{\Sigma}> 0$. We now go further, and show that within the set of Neumann free boundary support surfaces not satisfying $H_{\Sigma}> 0$, $\calS$, the set of surfaces for which $(\ref{mainresulteqn})$ does not hold in general is dense in $\calS$. This shows that $H_{\Sigma}> 0$ is an appropriate condition to place on the flows. Of course, by taking minute copies of the already existing examples and gluing them onto a given support surface, again allows for cases where ($\ref{base}$) will fail while staying near the original support surface with respect to the Hausdorff metric. We therefore consider metrics ensuring that the support surfaces have greater similarity of structure in order to be considered close. That is, density is taken with respect to a metric based on the norms of the homeomorphisms between the support surfaces defined below.

\begin{Definition}\label{C3}
Let $\cl{S}$ denote the set of all Neumann free boundary support surfaces in $\R^{n+1}$. For $\Sigma \in \cl{S}$, let $\calN(\Sigma)$ denote the set of all mean curvature flows with Neumann free boundary conditions on the solid support surface $\Sigma$, $\cl{M}:=(M_t)_{t\in [0,T)}$, for which
$$\Hm{n}(sing_T\cl{M}\cap \Sigma)>0.$$
Define
$$\calS:=\cl{S}\sim\{\Sigma \in \cl{S}:H_{\Sigma}(x)> 0 \hbox{ for all }x\in \Sigma\} \hbox{ and}$$
$$\calS_0:=\{\Sigma \in \calS:\calN(\Sigma)\not=\emptyset\}.$$
Let $\Phi$ denote the set of $C^3$-diffeomorphisms
$$\{\phi:A\rightarrow B:A,B\subset \R^{n+1}\}.$$
Whenever $\phi_1,\phi_2\in \Phi$ satisfy $\phi_i:A\rightarrow B_i$ for some diffeomorphic manifolds, $B_1, B_2\subset \R^{n+1}$, diffeomorphic to some $A\subset \R^{n+1}$, define
$$||\phi_1-\phi_2||_k:=\sum_{0\leq |\alpha|\leq k}\sup_{x\in A}||D^{\alpha}\phi_1(x)-D^{\alpha}\phi_2(x)||$$
for $k\in \N$ and where $||\cdot||$ denotes the appropriate usual Euclidean distance. We define, otherwise, $||\phi_1-\phi_2||_k=\infty$.

Additionally, we define
$$||\phi_1-\phi_2||_{\cl{S}}:=||\phi_1-\phi_2||_1+\sup_{x\in A}|H_{\phi_1(A)}(\phi_1(x))-H_{\phi_2(A)}(\phi_2(x))|.$$

Finally, for $\Sigma\in \cl{S}$ and sets $\Sigma_1, \Sigma_2 \in \cl{S}$ diffeomorphic to $\Sigma$, define
$$d_{\Phi, \Sigma}(\Sigma_1,\Sigma_2):=\inf\{||\phi_1-\phi_2||_{\cl{S}}:\phi_i \in \Phi, \phi_i:\Sigma\rightarrow \Sigma_i, i\in \{1,2\}\}.$$
If $\Sigma_1$ or $\Sigma_2$ is not diffeomorphic to $\Sigma$, we define $d_{\Phi,\Sigma}(\Sigma_1, \Sigma_2)=\infty.$
\end{Definition}
\begin{Remark}
\begin{enumerate}[(1)]
\item We note that $d_{\Phi,\Sigma}(\cdot,\cdot)$ is a metric satisfying $d_{\Phi,\Sigma}\geq d_{\Hm{}}$, where $d_{\Hm{}}$ denotes the Hausdorff distance. $d_{\Phi,\Sigma}$ allows us to consider support surfaces that are close to a given support surface in 
a more natural sense than Hausdorff distance, with which more drastic changes to the geometry of the surface would be allowed. That is, with these metrics, gluing an extremely small but highly curved piece of surface to an otherwise nearly flat surface is still considered a large variation. In particular, gluing minute copies of the counterexample in \cite{koeller4} onto a given support surface will not, in general, result in a nearby surface.
\item $d_{\Phi,\Sigma}$ is bounded from above by the $C^k$ norms for diffeomorphisms for $k\geq 2$. This fact helps us obtain estimates for the distances between two support surfaces below.
\item In the case that there is an $M\subset \Sigma\subset \R^{n+1}$ with $A_1$ and $A_2$ diffeomorphic to $M$, we also write
$$d_{\Phi, \Sigma}(A_1,A_2):=\inf\{||\phi_1-\phi_2||_{\cl{S}}:,\phi_i \in \Phi, \phi_i:M\rightarrow A_i, i\in \{1,2\}\}.$$
\item The inclusion of the mean curvature in the norm is important, as it is sets whose definition is based on the curvature of the elements that is being observed. Without this element it would also be true that $\calS_0$ is dense in $\cl{S}$.
\end{enumerate}
\end{Remark}
For notational convenience in our theorem, we make the following nomenclaturial definition.
\begin{Definition}\label{graph}
Let $X$ be an affine plane in $\R^n$, $f: X\rightarrow X^{\perp}$ be a function, and $(x,y)$ denote $x+y$ for $x\in X$ and $y \in X^{\perp}$. We write $gf$ to denote the graph function of $f$. That is, for $x\in X$
$$gf(x):=(x,f(x)).$$
\end{Definition}
\begin{Theorem}\label{HSigpositive}
For any $\e>0$ and any $\Sigma \in \calS$, there exists a $\Sigma_0 \in \calS_0$ such that 
$$d_{\Phi,\Sigma}(\Sigma,\Sigma_0)<\e.$$
\end{Theorem}
\begin{proof}
Let $\Sigma$ satisfy the rolling ball condition for balls of radius $R_0>0$ for which $H_{\Sigma}(\tilde{x})\leq 0$ for some $\tilde{x}\in \Sigma$. By otherwise replacing $\Sigma$ with a very small variation ($<<\e$ with respect to $d_{\Phi, \Sigma}$) near $\tilde{x}$, by the same method as that described below, we may assume that $H_{\Sigma}(x)<0$ for some $x\in \Sigma$ near $\tilde{x}$. Moreover, by otherwise rotating and translating $\Sigma$, we can assume $x=0$ and $T_x\Sigma=\R^n$. Furthermore, on some small open ball,$B_{\rho_0}(0)\subset U \subset\R^n$, $\rho_0<<R_0$, we can write 
$$\Sigma=g \psi_{\Sigma} \ \ \hbox{ and } \ \ B_{\Sigma}^0=g \psi_B,$$
where $B_{\Sigma}^0:=B_{h}(0-(h e_{n+1}))$, $h:=(n|H_{\Sigma}(0)|)^{-1}$, and $\psi_{\Sigma}, \psi_B:\R^n\rightarrow \R$ are $C^3$-functions. We see that $\psi_{\Sigma}(0)=\psi_B(0)$, $D\psi_{\Sigma}(0)=D\psi_B(0)$, and $$H_{\Sigma}((g\psi_{\Sigma})(0))=H_{\Sigma}(0)=H_{B_{\Sigma}^0}(0)=H_{B_{\Sigma}^0}((g\psi_B)(0)).$$
As $\Sigma$ is a locally compact $C^3$-surface satisfying a rolling ball condition, it follows that for some small $\rho_2<\rho_1<\rho_0$, there is an $\eta\in C^3(B_{\rho_0}^n(0),\R)$ such that 
\begin{equation}
\eta=
\begin{cases}
\psi_B \ \ \hbox{ on } \ \ B_{\rho_2}^n(0) \cr
\psi_{\Sigma} \ \ \hbox{ on } \ \ B_{\rho_0}^n(0)\sim B_{\rho_1}^n(0)
\end{cases}
\end{equation}
and $$||\eta-\psi_{\Sigma}||_{\cl{S}}<\frac{\e}{2}.$$
Set 
$$\Sigma_1:=(\Sigma \sim g\psi_{\Sigma}|_{B_{\rho_0}(0)})\cup g \eta |_{B_{\rho_0}(0)}.$$
By the rolling ball condition on $\Sigma$ and the fact that $\rho_0<<R_0$, we can also choose the $\eta$ above in such a way that $\Sigma_1$ is a smooth hypersurface satisfying the rolling ball condition for balls of radius $0<R_1\leq R_0$. It follows that $\Sigma_1 \in \calS$ and by selecting $\phi_1$ to be the identity transformation and $\phi_2(x):=\eta(\pi_{\R^n}(x))$ in Definition $\ref{C3}$, it can be calculated that
\begin{equation}\label{absch1}
d_{\Phi, \Sigma}(\Sigma,\Sigma_1)\leq ||\phi_1-\phi_2||_{\cl{S}}\leq \e/2.
\end{equation}

Furthermore, $g \eta|_{B_{\rho_2}(0)} \subset \Sigma_1$ is a rotationally symmetric, smooth hypersurface with constant mean curvature.

Take now 
$$\rho_3<<\min\{\rho_2 , h , R_1\}$$ and consider $\psi_B^1:\R\rightarrow \R$ defined by 
$$\psi_B^1(x):=(h^2-x^2)^{1/2}-h$$
on $[-\rho_3,\rho_3]$, a piece of circle of identical radius, $h$, to $B_{\Sigma}^0(0)$. 

Let $B_1:=g\psi_B^1|_{[-\rho_3,\rho_3]}$, $x_i:=\psi_B^1((-1)^i\rho_3)$, $\theta_i:G(1,2)\rightarrow G(1,2)$ be a  rotation taking $\R$ to $T_{\psi^1_B(x_i)}B_1$ satisfying $\la \theta_i(e_{n+1}),e_{n+1}\ra \geq 0$, and let $\delta_0<<\rho_3$.
For $\delta>0$, define
$$y_{\delta}^1(z):=(\delta^2-(z-\delta)^2)^{1/2}, z\in [0,\delta]$$
and
$$y_{\delta}^2(z):=(\delta^2-(z+\delta)^2)^{1/2}, z\in [-\delta,0].$$
For $\delta<\delta_0$ and $a_1<\delta$ define 
$$Q_1:=\theta_1(y_{\delta}^1([0,\delta]))+x_1,$$
$$Q_2:=\theta_2(y_{\delta}^2([-\delta,0]))+x_2,$$
$$Q_1^{a_1}:=\theta_1(y_{\delta}^1([a_1,\delta]))+x_1, \ \ \hbox{ and}$$
$$Q_2^{a_1}:=\theta_2(y_{\delta}^2([-\delta,-a_1]))+x_2.$$
$Q_1$ and $Q_2$ are smooth hypersurfaces in $\R^2$ satisfying 
$$Q_i\subset\{(x,y)\in\R^2:|x|\leq \rho_2,y\geq \psi_B^1(x)\}$$
$$Q_i\cap \Sigma_1=x_i, \hbox{ and}$$
$$\la \nu_{Q_i}(x_i),\nu_{\Sigma_1}(x_i)\ra =0 \ \ i\in \{1,2\},$$
where $\nu_{Q_i}$ is the unit normal for $Q_i$ satisfying $\la v_{Q_i},e_{n+1}\ra >0$.

Furthermore, we can choose $a_1, a_2$, and $\delta$ with $a_1<\delta<\delta_0$, $a_2<a_1$, and smooth functions, $\phi_1, \phi_2 \in C^3(\R,\R)$, so that $$Q_1^{a_1}=g\phi_1([-\rho_3,a_2-\rho_3]), Q_2^{a_1}=g\phi_2([\rho_3-a_2,\rho_3]),$$
$$\la \nu_{Q_i}(g\phi_i((-1)^i\rho_3)), e_{n+1}\ra < \la\nu_{\Sigma}(g\psi_B^1((-1)^i\rho_3)),e_{n+1}\ra, $$
$$\la\nu_{\Sigma}(g\psi_B^1((-1)^i\rho_3)),e_{n+1}\ra < \la\nu_{\Sigma}\psi_B^1((-1)^i(\rho_3-a_2)),e_{n+1}\ra, $$
$$\phi_i((-1)^i\rho_3)<\psi_B^1(-\rho_3)+\delta_0= \psi_B^1(\rho_3)+\delta_0, \hbox{ and}$$
$$\phi_i((-1)^i(\rho_3-a_2))=\psi_B^1((-1)^i(\rho_3-a_2)) < \psi_B^1(x)+\delta_0$$
for all $x\in (a_2-\rho_3, \rho_3-a_2)$.

By construction, we also have $H_{Q_i}(x)=-\delta^{-1}<0$ for all $x\in Q_i$, $i\in \{1,2\}$. We can now take a $\psi_{M}^1:\R\rightarrow \R$, $\psi_{M}^1\in C^3((-\rho_3,\rho_3))$, such that 
\begin{equation}
\psi_{M}^1(x)=
\begin{cases}
\eta((x,0,...,0))+\delta_0 \ \ & x\in (-3\rho_3/4,3\rho_3/4) \cr
\phi_2(x) \ \ & x\in (\rho_3-a_2,\rho_3) \cr
\phi_1(x) \ \ & x\in (-\rho_3,a_2-\rho_3)
\end{cases}
\end{equation}
and such that $D^2\psi_{M}^1<0$. Let $M_0^1:=Q_1\cup Q_2 \cup g\psi_{M}^1$ and define $\psi_{M}^n:\R^n\rightarrow \R$ by 
$$\psi_{\Sigma}^n(x):=\psi_{M}^1(|x|).$$
Define further 
$$M_0:=\{(x_n,x)\in \R^{n+1}:(|x_n|,x)\in M_0^1\},$$
a $C^3$-hypersurface in $\R^n$. By the construction we see that 
$$\partial M_0=M_0 \cap \Sigma_1,$$
$$\la \nu_{M_0}(x),\nu_{\Sigma_1}(x)\ra=0 \ \ x\in \partial M_0,$$
$$\pi_{\R^n}^{-1}(B_{7\rho_3/8}^n(0))\cap M_0=g \psi_{M}^n(B_{7\rho_3/8}^n(0)),$$
and $H_{M_0}(x)<0$ for all $x\in M_0$. Thus $M_0$ is a rotationally symmetric initial surface to mean curvature flow with Neumann free boundary conditions on the Neumann free boundary support surface $\Sigma_1$ with negative mean curvature.

We calculate that 
\begin{equation}\label{absch2}
d_{\Phi,\Sigma}(\pi_{\R^n}^{-1}(B_{3\rho_3/4}^n(0))\cap M_0, g \eta|_{B_{3\rho_3/4}^n(0)}) \leq ||\psi_{\Sigma}^n-\eta||_{C^3(B_{3\rho_3/4}^n(0))}=\delta_0.
\end{equation}
By $(\ref{absch2})$, and using the additional fact that $\la\nu_{M_0}(x),e_{n+1}\ra>C_{M_0}>0$ for all $x\in M_0 \cap B_{3\rho_3/4}$ we can find a $T_0>0$ such that there is a solution to mean curvature flow with Neumann free boundary conditions $\cl{M}:=(M_t)_{t\in [0,T_0]}$ with $T_0<T$ (where $T$ is the first singular time), for which
\begin{enumerate}[(i)]
\item $H_{M_t}(x)<0$ for all $x\in M_t$, $t\in [0,T]$,
\item There is a function $\phi:B_{\rho_3/2}^n(0)\times [0,T] \rightarrow \R \in C^3$ such that 
$$\phi(x,t)=\{y\in M_t:\pi_{\R^n}(y)=x\}, \ \ \hbox{ and}$$ 
\item $||\phi(\cdot,t)-\eta||_{C^3(B_{\rho_3/2}^n(0))}<2\delta.$
\end{enumerate}
Write $I:=\{(x_n,x)\in \R^{n+1}:|x_n|\in B_{\rho_3/2}^n(0), \eta(x_n)<x<\phi(x_n,T)\}$. By (ii) and the fact that $(M_t)$ is a mean curvature flow supported on $\Sigma_1$, we deduce that 
$$\bigcup_{t\in [0,T_0]}M_t\cap I=\emptyset.$$
By (iii) we may now take a $\psi_T\in C^3(B_{\rho_3/2}^n(0))$ defined by
\begin{equation}
\psi_T(x)=
\begin{cases}
\phi(x,T_0) \ \ x\in B_{\rho_3/4}^n(0) \cr
\eta(x) \ \ x\in B_{\rho_3/2}^n(0)\sim B_{3\rho_3/8}^n(0)
\end{cases}
\end{equation}
such that 
\begin{equation}\label{absch3}
||\psi_T-\eta||_{C^3(B_{\rho_3/2}^n(0))}\leq C(\delta_0,\delta_0/\rho_3)<\e/2
\end{equation}
for sufficiently small $\delta_0$.
Take now
$$\Sigma_0:=(\Sigma_1\sim g \eta|_{B_{\rho_3/2}^{n}(0)})\cup g \psi_T|_{B_{\rho_3/2}^n(0)}.$$
We deduce, from $(\ref{absch1})$ and $(\ref{absch3})$, that
\begin{eqnarray}
d_{\Phi,\Sigma}(\Sigma_0,\Sigma) & \leq & d_{\Phi,\Sigma}(\Sigma_0,\Sigma_1)+d_{\Phi,\Sigma}(\Sigma_1,\Sigma) \nonumber \\
& \leq & ||\phi_T-\eta||_{C^1(B_{\rho_3/2}^n(0))} + d_{\Phi,\Sigma}(\Sigma_1,\Sigma) < \e. \nonumber
\end{eqnarray}
By (i), and since $\partial \cl{M}\cap (\eta(B_{\rho_3/2}^n(0)) \cup I =\emptyset$, $(\cl{M})_{t\in [0,T_0)}$ is also a solution to mean curvature flow with Neumann free boundary conditions supported on $\Sigma_0$. Using $\Sigma_0$ as the support surface, we relabel the flow $\cl{M}_0$. In $\cl{M}_0$, $M_{T_0}\cap \Sigma_0 \not= \partial M_{T_0}$, and therefore, $T_0$ is the first singular time for this flow. As 
$$sing_{T_0}\cl{M}_0 \supset M_{T_0}\cap \Sigma_0 \supset \psi_T(B_{\rho_3/4}^n(0)),$$
 we deduce that 
$$\Hm{n}(sing_{T_0}\cl{M}_0)\geq \Hm{n}(B_{\rho_3/2}^n(0))>0$$
and thus that $\Sigma_0\in \calS_0$.
\end{proof}

%=======================
% BIBLIOGRAPHY AND INDEX
%=======================
\begin{comment}

\end{comment}

%===============================
% BIBLIOGRAPHY AND INDEX AMSREFS
%===============================

%\printindex

\end{document}